\documentclass[10pt]{article}
\usepackage{amsmath, amsthm, amssymb}
\usepackage{mathtools}
\usepackage[linesnumbered,ruled,vlined]{algorithm2e}
\usepackage{authblk}
\usepackage{url}
\usepackage{hyperref}
\usepackage[letterpaper,margin=1.5in]{geometry}
\newtheorem{theorem}{Theorem}
\newtheorem{lemma}[theorem]{Lemma}
\newtheorem{definition}[theorem]{Definition}
\newtheorem{rem}[theorem]{Remark}

\newtheorem{observation}[theorem]{Observation}
\newtheorem{strategy}{Strategy}

\def\N{\mathbb{N}}
\def\R{\mathbb{R}}

\newcommand{\mak}{Maker}
\newcommand{\bre}{Breaker}
\newcommand{\MBG}{\mak-\bre\ game}
\newcommand{\q}{q}
\newcommand{\pot}{\textnormal{pot}}
\newcommand{\Pot}{\textnormal{POT}}
\newcommand{\epsa}{\gamma}
\newcommand{\epsb}{\epsilon}
\newcommand{\epsc}{\xi}
\newcommand{\epsd}{\delta}
\newcommand{\firsteps}{\epsilon^*}
\newcommand{\lam}{6\beta}
\newcommand{\numcrit}{c}
\newcommand{\crit}{\textnormal{crit}}
\newcommand{\diff}[1]{\Delta_{#1}}
\newcommand{\diffrest}[1]{r_{#1}}
\newcommand{\bigrest}{R}
\newcommand{\free}{f}
\newcommand{\bal}{\textnormal{bal}}
\renewcommand{\deg}{\textnormal{deg}}
\newcommand{\baldeg}{{\deg^*}}
\newcommand{\defi}{d}
\newcommand{\inc}[1]{I_{#1}}
\newcommand{\dec}[1]{D_{#1}}
\newcommand{\decfree}[1]{\dec{#1}^{\textnormal{free}}}
\newcommand{\decheads}[1]{\dec{#1}^{\textnormal{heads}}}
\newcommand{\dectails}[1]{\dec{#1}^{\textnormal{tails}}}
\newcommand{\ott}[1]{\dec{#1}^{\textnormal{+}}}
\newcommand{\utt}[1]{\dec{#1}^{\textnormal{--}}}
\newcommand{\deczero}[1]{\dec{#1}^{0}}

\newcommand{\critdiff}[1]{\Delta_{#1}}
\newcommand{\restdiff}[1]{r_{#1}}
\newcommand{\first}[1]{{#1}^{(1)}}
\newcommand{\tmin}{{t^*}}

\newcommand{\ceiling}[1]{\left\lceil{#1}\right\rceil}

\title{A new Bound for the Maker-Breaker Triangle Game}

\author[1]{Christian Glazik}
\author[1]{Anand Srivastav}

\affil[1]{Department of Mathematics, Kiel University\\
  Christian-Albrechts-Platz 4\\
24118 Kiel, Germany;  \{glazik,srivastav\}@math.uni-kiel.de}

\date{}
\begin{document} 
 
\maketitle
 
\begin{abstract}
The triangle game introduced by Chv\'{a}tal and Erd\H{o}s (1978)
is one of the most famous combi\-natorial games.
For $n,\q\in\N$, the $(n,\q)$-triangle game is played
by two players, called Maker and Breaker,
on the complete graph $K_n$. 
Alternately Maker claims one edge and thereafter Breaker claims $\q$ edges of the graph.
Maker wins the game if he can claim all three edges of a triangle,
otherwise Breaker wins.
Chv\'{a}tal and Erd\H{o}s (1978)
proved that for $\q<\sqrt{2n+2}-5/2\approx 1.414\sqrt{n}$
Maker has a winning strategy,
and for $\q\geq 2\sqrt{n}$ Breaker has a winning strategy.
Since then, the problem of finding the exact leading constant
for the threshold bias of the triangle game has been one of the
famous open problems in combinatorial game theory.
In fact, the constant is not known for any graph with a cycle and we
do not even know
if such a constant exists.
 Balogh and Samotij (2011) slightly improved the
Chv\'{a}tal-Erd\H{o}s constant
  for Breaker's winning strategy from $2$ to
  $1.935$ with a randomized approach. Since then no progress was
made.
  In this work, we present a new deterministic strategy for
Breaker's win whenever $n$ is sufficiently large and
$\q\geq\sqrt{(8/3+o(1))n}\approx 1.633\sqrt{n}$,
significantly reducing the gap towards the lower bound.
In previous strategies Breaker chooses his edges
such that one node is part of the last edge chosen by Maker,
whereas the remaining node is chosen more or less arbitrarily.
In contrast, we introduce a suitable potential function on the set of nodes.
This allows Breaker to pick edges that connect the most `dangerous' nodes.
The total potential of the game may still increase, 
even for several turns, but finally Breaker's strategy
prevents the total potential of the game from exceeding
a critical level and leads to Breaker's win.
\end{abstract}
\pagebreak
\section{Introduction}
For $n,\q\in\N$ the $(n,\q)$-triangle game
is played on the complete graph $K_n$.
In every turn Maker claims an unclaimed edge from $K_n$,
followed by Breaker claiming $\q$ edges.
The game ends when all edges are claimed either by Maker or Breaker.
If Maker manages to build a triangle, he wins the game,
otherwise Breaker wins.
This game is one of the most prominent examples
of \emph{\MBG s},
in which Maker tries to build a certain structure
while Breaker tries to prevent this.
These are games of perfect information without chance moves,
so either Maker or Breaker has a winning strategy,
in which case we say the game is Maker's win or Breaker's win, respectively.
For more information on \MBG s we refer to
the paper by Krivelevich~\cite{Krivelevich14}.
\begin{figure}[h]
  \centering
  \fbox{
    \begin{minipage}{.217\textwidth}
      \centering
      \includegraphics[scale=0.5]
      {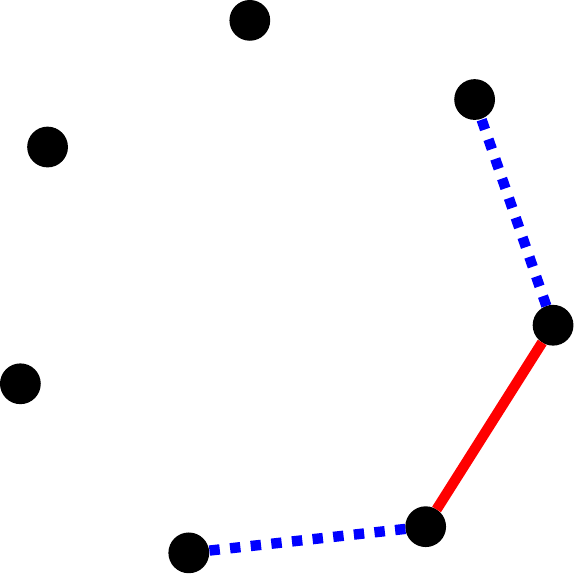}
      
    \end{minipage}}
  \fbox{
    \begin{minipage}{.217\textwidth}
      \centering
      \includegraphics[scale=0.5]
      {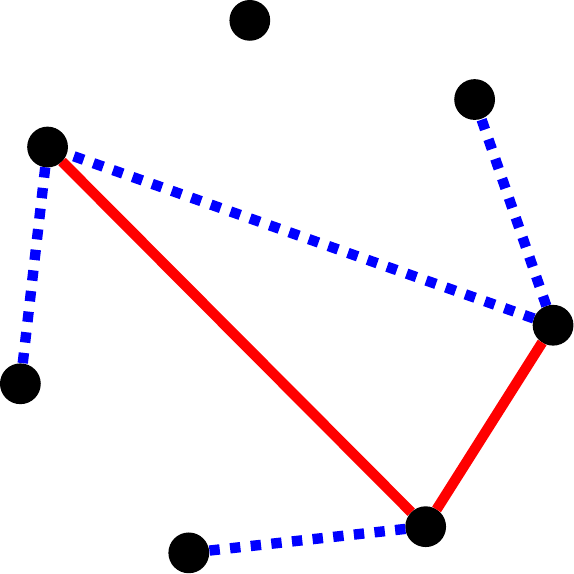}
    \end{minipage}}
  \fbox{
    \begin{minipage}{.217\textwidth}
      \centering
      \includegraphics[scale=0.5]
      {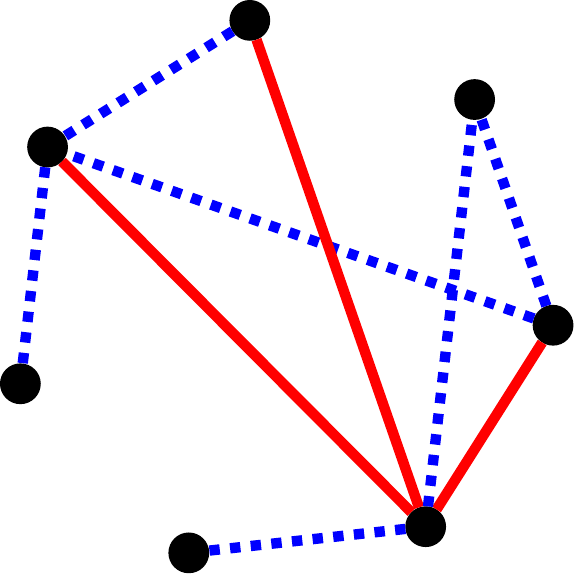}
    \end{minipage}}
  \fbox{
    \begin{minipage}{.217\textwidth}
      \centering
      \includegraphics[scale=0.5]
      {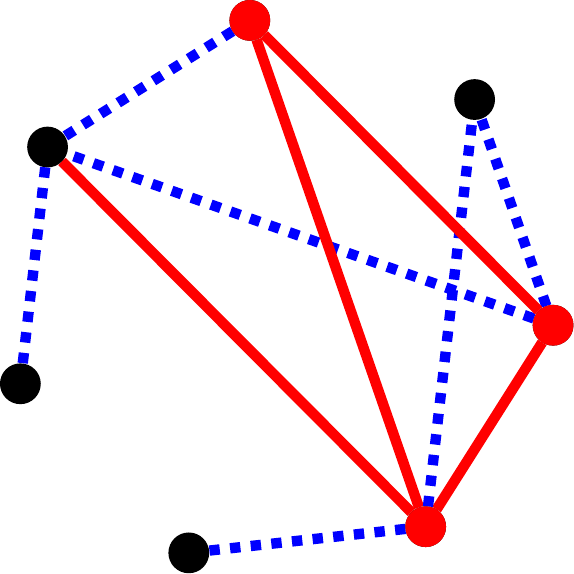}
    \end{minipage}}
  \caption{The $(7,2)$-triangle game is a \mak's win. \mak-edges are red, \bre-edges blue.}
  \label{fig:degree}
\end{figure}
\subsection{Previous work}
\MBG s have been extensively studied by Beck~\cite{Beck81,Beck82,Beck85},
concerning, e.g., games in which Maker tries to build
a Hamiltonian cycle, a spanning tree or a big star.
Beck~\cite{Beck82} also presented very general sufficient conditions for Maker's win
and Breaker's win.
In his work he generalized the Erd\H os-Selfridge Theorem~\cite{Erdoes73},
which gives a winning criterion for Breaker for the case $\q=1$.
A direct application of these criteria to specific games as the triangle game
often does not lead to strong results.
However, it turned out to be a powerful tool, e.g.\
Bednarska and \L uczak~\cite{Bednarska00} used it to prove the following fundamental result:
For a fixed graph $G$ consider the $(G,n,\q)$-game,
in which Maker has to build a copy of $G$. There exist constants $c_0,C_0$
such that the game is a Maker's win for $\q\leq c_0n^{1/m(G)}$
and a Breaker's win for $\q\geq C_0n^{1/m(G)}$,
where $m(G):=\max\Big\{\frac{e(H)-1}{v(H)-1}:H\subseteq G,v(H)\geq 3\Big\}$.
This result recently was further generalized for hypergraphs by
Kusch et al.~\cite{Kusch17}.
Bednarska and \L uczak also conjectured that $c_0$ and $C_0$
can be chosen arbitrarily close to each other.
Until today this conjecture couldn't be proved or disproved for any 
fixed graph $G$ that contains a cycle.\par
The special case of the triangle game was proposed by
Chv\' atal and Erd\H os~\cite{Chvatal78},
who presented a winning strategy for Maker
if $\q<\sqrt{2n+2}-5/2$,
and a winning strategy for Breaker
if $\q\geq 2\sqrt{n}$.
Both strategies are rather simple:
If $\q<\sqrt{2n+2}-5/2$, Maker can win by
fixing a node $v$ and then simply claiming all his edges incident to $v$.
At some point of time Breaker will not be able
to close all Maker paths of length~$2$, so Maker can complete such a path
to build a triangle.
If $\q\geq 2\sqrt{n}$, Breaker can always close all paths of length~$2$
created by Maker and at the same time
prevent Maker from building a star of size $\q/2$.
To achieve this he first closes all new Maker paths of length~$2$
and then claims arbitrary edges
such that at the end of the turn he claimed exactly
$\sqrt{n}$ edges incident in $u$
and the remaining $\sqrt{n}$ edges incident in $v$,
where $\{u,v\}$ is the edge recently claimed by Maker.\par
Chv\' atal and Erd\H os also asked for the \emph{threshold bias}
of this game, i.e., the value $\q_0(n)$ such that the game is Maker's win
for $\q\leq \q_0(n)$ and Breaker's win for $\q>\q_0(n)$.
For the triangle game, the asymptotic order of $\Theta(\sqrt{n})$
follows directly from the two strategies given above
and also occurs as special case of the result by Bednarska and \L uczak,
but the gap between $\sqrt{2n}$ and $2\sqrt{n}$ could not be narrowed
for many years.
In 2011, Balogh and Samotij~\cite{Balogh11} presented a randomized Breaker strategy,
improving the lower bound for Breaker's win 
to about $(2-1/24)\sqrt{n}\approx 1.958\sqrt{n}$.
\subsection{Our Contribution}
In our work we present a new deterministic strategy for Breaker
that further improves the recent lower bound for Breaker's win
to $\q=\sqrt{(8/3+o(1))n}\approx 1.633\sqrt{n}$,
assuming $n$ to be sufficiently large.
The global idea of our strategy is as follows.
Instead of claiming arbitrary edges incident
in the nodes of the last edge claimed by Maker,
as done in the strategy of Chv\' atal and Erd\H os,
Breaker claims only edges that connect the `most dangerous' nodes,
i.e., nodes that already have many incident Maker edges
and rather few Breaker edges.
Proceeding this way, Breaker needs fewer edges
to prevent Maker from building a $\q/2$-star. 
For the realization of this idea
we use an (efficiently computable) 
\emph{potential function} to decide which edges are
most dangerous and should be claimed next to prevent Maker
from building any triangle \emph{or} big star.
In contrast to Beck~\cite{Beck82}, instead of assigning a potential
to every winning set, our potential function is defined directly on the set of nodes.
However, the most significantly difference to Beck and other previous potential-based 
approaches is that our potential function 
is not necessarily decreasing in every single turn.
Some \emph{critical} turns may occur in which the potential increases, 
so the challenge is to bound the number and the impact of these critical rounds.
This new approach requires plenty of analytic work
but turns out to be a more powerful technique than classic potential-based approaches
and also might be of interest for other kinds of \MBG s. 
\section{\bre's strategy}
We start by introducing the potential function
which forms the basis for \bre's strategy.
During the game, denote by $M$ the \mak\ graph
consisting of all edges claimed by \mak\ so far
and let $B$ denote the corresponding \bre\ graph.
For $v\in V$ and $H\in\{M,B\}$ let $\deg_H(v)$ denote
the degree of $v$ in $H$.
For a turn $t$, $\deg_{H,t}(v)$ denotes the degree
of $v$ in $H$ directly after turn $t$.
\subsection{The potential function}\label{sec:potfunc}
Let $\firsteps>0$ and $\beta=\frac{8}{3}+\firsteps$.
In this chapter we consider the $(n,\q)$-Triangle game
with $\q=\sqrt{\beta n}$.
As mentioned in the introduction,
for $\beta\geq 4$ there exist known winning strategies
for \mak, so we will assume $\beta\leq 4$ if necessary.
Fix $\epsd\in(0,1-\frac{8}{3\beta})$.
\begin{definition}\label{def:pot}
  For every $v\in V$ define the \emph{balance} of $v$ as
  \[\bal(v):=\frac{8(n-\deg_B(v))}
    {\q^2(1-\epsd)(3+\epsd)-4\deg_M(v)(2\q-\deg_M(v))}.\]
  Moreover we define $p_0$ as the balance of a node in the very beginning of the game,
  i.e.\
  \[p_0:=\frac{8n}{\q^2(1-\epsd)(3+\epsd)}=\frac{8}{\beta(1-\epsd)(3+\epsd)}.\]
\end{definition}
The balance of a node is a measure of the ratio of \mak- and \bre-edges
incident in this node: The more \mak-edges and the fewer \bre-edges incide
in $v$, the bigger the balance value gets.
A detailed interpretation of the balance value can be found in Section~\ref{sec:interpretation}.
\par
For the success of \bre's strategy it is crucial that $p_0<1$
(e.g.\ for the choice of $\eta$ in Section~\ref{sec:crit_turns}).
This is assured by the next remark.
\begin{rem}\label{rem:p0}
  It holds $\frac{8}{3\beta}<p_0<\frac{8}{3\beta(1-\epsd)}<1$.
\end{rem}
\begin{proof}
  The second and third inequality follow directly
  from $\epsd\in(0,1-\frac{8}{3\beta})$.
  For the first inequality,
  note that $(1-\epsd)(3+\epsd)=3-2\epsd-\epsd^2<3$,
  so we get $p_0=\frac{8}{\beta(1-\epsd)(3+\epsd)}>\frac{8}{3\beta}$.
\end{proof}
During the game, \bre\ will not be able
to keep all nodes at their start balance.
Some nodes will get more \bre-edges than needed, others less.
This \emph{deficit} of a node will be used to define its potential.
\begin{definition}\label{def:pot2}
  Consider the game at an arbitrary point of time.
  For a node $v\in V$ let \mbox{$\baldeg(v)\in\R$}
  be the \emph{balanced \bre-degree} of this node,
  i.e.\ the \bre-degree that would be necessary, so that $\bal(v)=p_0$.
  Formally we define
  \[\baldeg(v):=
    n-p_0\left(\frac{\q^2(1-\epsd)(3+\epsd)}{8}
      -\deg_M(v)\left(\q-\frac{\deg_M(v)}{2}\right)\right).
  \]
  The \emph{deficit} of $v$ is defined by
  \[\defi(v):=\baldeg(v)-\deg_B(v).\]
  Finally, let $\mu:=1+\frac{6\beta\ln(n)}{\epsd \q}$.
  Define the \emph{potential} of $v$ as
  \[\pot(v):=\begin{cases}0&\textnormal{if }\deg_M(v)+\deg_B(v)=n-1\\
      \mu^{\defi(v)/\q}&\textnormal{else}\end{cases}\]
  and for an unclaimed edge $e=\{u,w\}$ define the \emph{potential} of $e$
  as
  $\pot(e):=\pot(u)+\pot(w)$.
  For every turn $t$ we define $\pot_t(v)$ ($\pot_t(e)$, resp.)
  as the potential of $v$ ($e$, resp.)
  directly after turn $t$ and $\pot_0(v)$ as the potential of
  $v$ at the beginning of the game.
  Analogously we define 
  $\baldeg_t(v)$ and $\defi_t(v)$.
  The \emph{total potential} of a turn $t$ is defined as
  $\Pot_t:=\sum_{v\in V}\pot_t(v)$.
  The \emph{total starting potential} is defined as
  $\Pot_0:=\sum_{v\in V}\pot_0(v)$.
\end{definition}
\begin{lemma}\label{lem:start_potential}
  The total starting potential fulfills $\Pot_0=n$.
\end{lemma}
\begin{proof}
  Let $v\in V$ with $\deg_M(v)=\deg_B(v)=0$.
  Then,
  \begin{align*}
    \baldeg(v)
    =n-p_0\left(\frac{\q^2(1-\epsd)(3+\epsd)}{8}\right)
    =n-p_0\cdot n\cdot p_0^{-1}=0.
  \end{align*}
  This implies $\pot(v)=\mu^{\defi(v)/\q}=\mu^{(\baldeg(v)-\deg_B(v))/\q}=\mu^0=1$,
  so \[\Pot_0=\sum_{v\in V}\pot_0(v)=\sum_{v\in V}1=n.\]
\end{proof}
\bre's aim is to keep the total potential as low as possible.
The next lemma ensures that if \bre\ can keep the potential of every single node
below $2n$, he can prevent \mak\ from raising the \mak-degree
of a node above $\q/2$.
We will later show (Theorem~\ref{thm:smallpot}) that \bre\ is even able
to keep the \emph{total} potential of the game below $2n$. 
\begin{lemma}\label{lem:potentialvsdegree}
  If $n$ is sufficiently big, for every turn $t$ and every node $v\in V$
  the following holds:
  \[0<\pot_t(v)\leq 2n\Rightarrow \deg_{M,t}(v)\neq\left\lceil \q/2\right\rceil-1.\]
\end{lemma}
\begin{proof}
  Let $t$ be a turn and $v\in V$ with $\pot_t(v)>0$ and $\deg_{M,t}(v)=\left\lceil \q/2\right\rceil-1$.
  We show that $\pot_t(v)>2n$.
  Because $\pot_t(v)\neq 0$, we have $\pot_t(v)=\mu^{\defi_t(v)/\q}$.
  We claim (and later prove) that
  \begin{equation}\label{claim}
    \defi_t(v)\geq\frac{2\epsd n}{3}.
  \end{equation}
  This implies
  \begin{align*}
    \pot_t(v)
    &=\mu^{\defi_t(v)/\q}
      \geq \mu^{2\epsd n/3\q}
      =\left(1+\frac{\lam \ln(n)}{\epsd \q}\right)^{2\epsd n/3\q}\\
    &=\left(1+\frac{\lam \ln(n)}{\epsd \q}\right)
      ^{\left(\frac{\epsd \q}{\lam \ln(n)}+1\right)\left(\frac{\epsd \q}{\lam \ln(n)}+1\right)^{-1}\frac{2\epsd n}{3\q}}
      \geq e^{\alpha},
  \end{align*}
  where 
  \begin{align*}
    \alpha
    &=\left(\frac{\epsd \q}{\lam\ln(n)}+1\right)^{-1}\frac{2\epsd n}{3\q}
      =\left(\frac{\epsd \q\mu}{\lam\ln(n)}\right)^{-1}\frac{2\epsd n}{3\q}
      =\frac{4\beta n\ln(n)}{\q^2\mu}
      >2\ln(n),
  \end{align*}
  where for the last inequality we used that $\mu<2$ if $n$ is big enough.
  Finally we get $\pot_t(v)\geq e^\alpha>n^2$
  and for $n\geq 2$ this is at least $2n$.
  \par
  We still have to prove claim~\eqref{claim}.
  Recall that $\defi_t(v)=\baldeg_t(v)-\deg_{B,t}(v)$.
  We estimate $\baldeg_t(v)$ as
  \begin{align*}
    \baldeg_t(v)
    &=n-p_0\left(\frac{\q^2(1-\epsd)(3+\epsd)}{8}
      -\deg_{M,t}(v)\left(\q-\frac{\deg_{M,t}(v)}{2}\right)\right)\\
    &\geq n-p_0\left(\frac{\q^2(1-\epsd)(3+\epsd)}{8}
      -\left(\frac{\q}{2}-1\right)\left(\q-\frac{\q}{4}\right)\right)\\
    &=n+p_0\left(\q^2\left(\frac{3-(1-\epsd)(3+\epsd)}{8}\right)
      -\frac{3\q}{4}\right)\\
    &= n+p_0\bigg(\frac{\q^2\epsd}{4}+\q\bigg(\underbrace{\frac{\epsd^2\q}{8}-\frac{3}{4}}_{\geq 0 \textnormal{ if } n \textnormal{ suff. big}}\bigg)\bigg)\\
    &\geq n+\frac{p_0\beta \epsd n}{4}
      \geq n+\frac{2\epsd n}{3}. \tag{Remark~\ref{rem:p0}} 
  \end{align*}
  Therefore,
  \begin{align*}
    \defi_t(v)
    =\baldeg_t(v)-\deg_{B,t}(v)
    \geq  n+\frac{2\epsd n}{3}-n
    = \frac{2\epsd n}{3}.
  \end{align*}

\end{proof}

\subsection{The detailed strategy}\label{sec:strategy}
The basic idea is that \mak's task to build a triangle
is closely related to the task of connecting big stars.
Assume that at any time during the game
\mak\ manages to build a path $(u,v,w)$ of length $2$.
Then \bre\ is forced to immediately \emph{close} this path
by claiming the edge $\{u,w\}$
if he doesn't already own this edge.
So every sensible \bre-strategy will follow the simple rule
of immediately closing all \mak-paths of length~$2$.
Hence, the only chance for \mak\ to win the game
is to construct more than $\q$ paths of length~$2$ in a single turn,
so that \bre\ can't claim enough edges
to close all of them immediately.
By claiming an edge $\{u,v\}$,
\mak\ is building $\deg_M(u)+\deg_M(v)$ new paths of length~$2$.
This implies that if \bre\ at each turn closes all \mak-paths of length~$2$
and simultaneously manages to prevent \mak\ from building a $\q/2$-star,
he will win the game.
\begin{strategy}\label{strat:ours}
  Consider an arbitrary turn $t$.
  Let $e_M=\{u,v\}$ be the edge claimed by \mak\ in this turn.
  \bre's moves for this turn are split into two parts.
  \par
  \textbf{Part 1: closing paths.}
  \bre\ claims $\deg_{M,t-1}(v)$ edges incident in $u$
  and $\deg_{M,t-1}(u)$ edges incident in $v$ to close all new \mak-paths
  of length~$2$.
  If such a path is already closed,
  he claims an arbitrary edge incident in $u$ ($v$, resp.) instead.
  If all edges incident in $u$ ($v$, resp.) are already claimed,
  we call the turn $t$ an \emph{isolation turn}.
  In this case, \bre\ claims arbitrary unclaimed edges instead.
  We call the edges claimed during Part~1 \emph{closing edges}.
  $u$ ($v$, resp.) is called the \emph{head} of the closing edge,
  whereas the corresponding second node of the edge is called its \emph{tail}.
  \par
  \textbf{Part 2: free edges.}
  If after part 1 \bre\ still has edges left to claim
  (we will later show that this is always the case),
  he iteratively claims an edge $e$ with
  $\pot(e)\geq\pot(e')$ for all unclaimed edges $e'$,
  until he claimed all of his $\q$ edges.
  We call the edges claimed in Part~2 \emph{free edges}.
  The number of free edges claimed in turn $t$ is denoted by $\free(t)$.
  Note that 
  \begin{equation}\label{eq:freeedges}
   \free(t)=\q-\deg_{M,t-1}(u)-\deg_{M,t-1}(v).
  \end{equation}

\end{strategy}
Part~1 of the strategy is more or less obligatory,
because a \mak-path of length~$2$ that is not closed by \bre\
can be completed to a triangle in the next turn.
Part~2 is more interesting.
Our aim in the following sections is to prove Theorem~\ref{thm:final},
where we show that part~2 of the strategy
prevents \mak\ from building a $\q/2$-star,
so that \bre\ wins the game.\par
%
\begin{observation}\label{obs:isolation}
 We can assume that the game contains no isolation turns.
\end{observation}
\begin{proof}
 Consider an arbitrary isolation $t$ turn in the game, i.e.,
 a turn, after which one of the nodes of the edge $e_M$
 claimed in this turn by \mak\ has no unclaimed incident edges left.
 Right after the turn, every triangle $e_M$ belongs to is already blocked by
 \bre, so the edge $e_M$ is of no use for \mak\ from this time on.
 \bre\ even could pretend that the edge $e_M$ belongs to his own edges,
 so that in the turn $t$ \bre\ claimed $\q+1$ edges and \mak\ didn't claim any edge.
 Hence, a perfectly playing \mak\ will always try to avoid isolation turns.
 If he can't, he will definitely loose the game,
 since he can only claim useless edges until the end of the game.
\end{proof}

The following observation states that,
as long as \bre\ can keep the total potential below $2n$,
he will have at least $2$ free edges in every turn.
\begin{observation}\label{obs:freeedges}
 For every turn $t$ with $\free(t)\leq 1$ there exists a turn $t'<t$
 with $\Pot_{t'}> 2n$.
\end{observation}
\begin{proof}
 Let $t$ be a turn with $\free(t)\leq 1$ and let $\{u,v\}$ be the \mak-edge 
 of this turn.
 Because $\free(t)=\q-\deg_{M,t-1}(u)-\deg_{M,t-1}(v)$,
 we get $\deg_{M,t-1}(u)+\deg_{M,t-1}(v)\geq q-1$, 
 so there exists $w\in\{u,v\}$ with 
 $\deg_{M,t-1}(w)\geq \ceiling{\tfrac{q-1}{2}}\geq \ceiling{\tfrac{q}{2}}-1$.
 Hence, there exists a turn $t'\leq t-1$ with $\deg_{M,t'}(w)=\ceiling{\tfrac{q}{2}}-1$ and $\pot_{t'}(w)>0$. We apply Lemma~\ref{lem:potentialvsdegree}
 and get $\pot_{t'}(w)>2n$, so especially $\Pot_{t'}>2n$.
\end{proof}

\subsection{Main results}
In this subsection we prove that Strategy~\ref{strat:ours} works correctly
and is a winning strategy.
For both theorems in this subsection we assume that \bre\ plays according to Strategy~\ref{strat:ours}.
We further assume that 
$\q=\sqrt{(\tfrac{8}{3}+\firsteps)n}$ for some $\firsteps>0$ as stated above
and that $n$ is sufficiently large. 
For \bre's strategy it is crucial that the potential of every node is kept
below a certain level. This is ensured by the following theorem.
\begin{theorem}\label{thm:smallpot}
  For every turn $s$ it holds $\Pot_s< 2n$.
\end{theorem}
The proof of this theorem is the mathematical core of this paper and is given 
in the next section. The main result of our work is:
\begin{theorem}\label{thm:final}
  At the end of the game there exists no node
  with \mak-degree of at least $\q/2$ and \bre\ wins the game.
\end{theorem}
\begin{proof}
  Assume that there exists a node $v$ with $\deg_M(v)\geq \q/2$
  at the end of the game.
  Then, $\deg_M(v)\geq\lceil \q/2 \rceil$.
  Let $t$ denote the turn in which \mak\ claimed his
  \mbox{$\lceil \q/2\rceil$-th}
  edge incident in $v$, so $\deg_{M,t-1}(v)=\lceil \q/2\rceil -1$.
  Due to Theorem~\ref{thm:smallpot} we know that
  $\pot_{t-1}(v)\leq\Pot_{t-1}<2n$.
  Note that after turn $t-1$ there are still
  unclaimed edges incident in $v$, so $\pot_{t-1}(v)>0$.
  We apply Lemma~\ref{lem:potentialvsdegree} and get
  $\deg_{M,t-1}(v)\neq \lceil \q/2\rceil-1$, a contradiction.\par 
  So with every edge $\{u,v\}$ that \mak\ chooses he creates less 
  than $\deg_M(u)+\deg_M(v)<q$ new \mak-paths of length~$2$.
  Hence, \bre\ always has enough edges to close all \mak-paths of length~$2$
  and finally wins the game.
\end{proof}
\section{Analysis}
\subsection{Outline of the proof}
We proceed to prove Theorem~\ref{thm:smallpot}.
As it is depending on a series of lemmas,
for the reader's convenience
we first outline the argumentation in an informal way.
We distinguish two types of turns.
A turn is called \emph{non-critical},
if a certain fraction of the \bre-edges in this turn suffices
to compensate the total potential increase caused by \mak\ in this turn.
Otherwise, we call it \emph{critical}.
We start with an arbitrary critical turn $t_0$
in which the potential exceeds $n$.
Lemma~\ref{lem:critical_turns} gives us a useful characterization of critical turns.
This enables us to prove Theorem~\ref{thm:no_increase},
where we state that before a constant number of additional critical turns is played,
the total potential will sink below $n$ again.
Because a constant number of critical turns
cannot increase the total potential considerably much (Lemma~\ref{lem:crit_increase}),
we can prove that the total potential of the game never exceeds $2n$.
\subsection{Potential change in a single turn}
To analyze the potential change of a single turn,
we first present a few tools for estimation
of potential change caused by single \mak- and \bre-edges.
The next lemma shows how the addition of a single \mak-edge
changes the deficit of a node.
\begin{lemma}\label{lem:deficit_single}
  Consider an arbitrary point of time in the game.
  Let $u\in V$ and let
  ${\baldeg}'(u),\deg_M'(u)$ and $\defi'(u)$ be the
  balanced \bre-degree, \mak-degree and deficit of $u$
  after an additional edge incident in $u$ was claimed by \mak.
  Then,
  \[\defi'(u)-\defi(u)={\baldeg}'(u)-\baldeg(u)\leq p_0(\q-\deg_M(u)).\]
\end{lemma}
\begin{proof}
  The equation follows from the fact
  that an additional \mak-edge does not change $\deg_B(u)$.
  Using that $\deg_M'(u)=\deg_M(u)+1$ we continue
  \begin{align*}
    &{\baldeg}'(u)-\baldeg(u)\\
    &= p_0\deg_M'(u)\left(\q-\frac{\deg_M'(u)}{2}\right)
      -p_0\deg_M(u)\left(\q-\frac{\deg_M(u)}{2}\right)\\
    &= p_0\left(\deg_M(u)\left(\q-\frac{\deg_M(u)+1}{2}\right)
      +\left(\q-\frac{\deg_M(u)+1}{2}\right)\right)\\
    &-p_0\deg_M(u)\left(\q-\frac{\deg_M(u)}{2}\right)\\
    &=p_0(\q-\deg_M(u)-1/2)
      \leq p_0(\q-\deg_M(u)).
  \end{align*}
\end{proof}

\begin{lemma}\label{lem:pot_single}
\begin{itemize}
 \item[(i)] A single edge $e_M$ claimed by \mak\ increases the potential of a node
  by at most a factor of $\mu$
  and causes a total potential increase of at most $(\mu-1)\pot(e_M)$
  (where $\pot(e_M)$ denotes the potential of $e_M$
  when claimed by \mak).
 \item[(ii)] A single edge $e_B$ claimed by \bre\ causes a total
  potential decrease of at least $(1-\mu^{-1/\q})\pot(e_B)$
  (where $\pot(e_B)$ denotes the potential of $e_B$
  when claimed by \bre).
\end{itemize}
\end{lemma}
\begin{proof}
  (i). Let $e_M=\{u,v\}$. 
  For $w\in V$ let $\pot(w)$ denote the potential of $w$ before
  \mak\ claimed $e_M$ and $\pot'(w)$ denote the potential of $w$
  directly after \mak\ claimed $e_M$.
  If $e_M$ is not incident in $w$, the potential of $w$
  remains unchanged. If $e_M$ is the last unclaimed edge incident in $w$,
  $\pot'(w)=0$ and we are done.
  Otherwise we can apply Lemma~\ref{lem:deficit_single}
  and Remark~\ref{rem:p0} and get
  \begin{equation*}
    \frac{\pot'(w)}{\pot(w)}=\mu^{(\defi'(w)-\defi(w))/\q}
    \leq \mu^{p_0(\q-\deg_M(w))/\q}\leq \mu.
  \end{equation*}
  Because $e_M$ only changes the potential of $u$ and $v$,
  the total potential increase is
  $\pot'(v)-\pot(v)+\pot'(u)-\pot(u)\leq (\mu-1)\pot(e_M)$.\par
  (ii). Let $e_B=\{u,v\}$.
  Because $e_B$ only changes the potential of $u$ and $v$,
  the total potential decrease caused by $e_B$ is
  $\pot(v)-\pot'(v)+\pot(u)-\pot'(u)$,
  where $\pot(w)$ denotes the potential of $w$ before
  \bre\ claimed $e_B$ and $\pot'(w)$ denote the potential of $w$
  directly after \bre\ claimed $e_B$.
  We show that
  \[\pot(v)-\pot'(v)\geq(1-\mu^{-1/\q})\pot(v).\]
  Because the same holds for $u$, the claim~(ii) follows.
  If $e_B$ is the last unclaimed edge in $v$, $\pot'(v)=0$.
  Otherwise,
  \begin{equation*}
    \frac{\pot'(v)}{\pot(v)}=\mu^{(\defi'(u)-\defi(u))/\q}
    =\mu^{-1/\q},
  \end{equation*}
  where the last equation follows from the fact that
  a \bre-edge does not change $\baldeg(v)$ and increases $\deg_B(v)$ by $1$.
\end{proof}

Every turn $t$ starts with a \mak\ move,
i.e.\ an edge $\{u,v\}$ being claimed by \mak\
followed by $\q$ \bre\ moves.
While the \mak\ move causes a potential increase,
\bre's moves cause a decrease.
For every node $w\in V$, we denote its potential increase by $\inc{t}(w)$
and its potential decrease by $\dec{t}(w)$.
Note that every claimed edge only changes the potential
of its two incident nodes.
When following \bre's strategy,
there are four possible ways of potential decrease for the node $w$:
decrease caused by free edges, denoted by $\decfree{t}(w)$
and decrease caused by closing edges,
either $w$ being their head, denoted by $\decheads{t}(w)$,
or their tail, denoted by $\dectails{t}(w)$.
In the special case in which \mak\ or \bre\ claim the last unclaimed edge incident in $w$, 
the potential of $w$ is set to $0$, which causes an additional potential decrease. For technical reasons, this additional decrease is considered seperately and denoted by $\deczero{t}(w)$.
If for example \bre\ claims a free edge that is the last unclaimed edge incident in $w$, this edge contributes both to $\decfree{t}(w)$ and $\deczero{t}(w)$. For the contribution to $\decfree{t}(w)$ we only 
compute the potential change caused by the change of the balance value
and for the contribution to $\deczero{t}(w)$ we take the real potential decrease caused by the edge and subtract the computed contribution to $\decfree{t}(w)$.
Moreover, we further split $\decheads{t}(w)$ into two parts
$\decheads{t}(w)=\utt{t}(w)+\ott{t}(w)$, where 
\[\utt{t}(w):=\min\{\inc{t}(w),\decheads{t}(w)\}
  \ \textnormal{ and }\
  \ott{t}(w):=\max\{\decheads{t}(w)-\inc{t}(w),0\}.\]
If \mak\ claims an edge that connects two nodes with a very high \mak-degree,
it might happen that $\decheads{t}(w)>\inc{t}(w)$ for one or both of the newly connected nodes.
Otherwise, $\ott{t}=0$ and $\utt{t}(w)=\decheads{t}(w)$.\par
If for one of these values we omit the argument,
we always mean the total potential increase (decrease) added up over all nodes.
For example, $\inc{t}:=\sum_{v\in V}\inc{t}(v)$.
For every turn $t$ we have
\[\Pot_t-\Pot_{t-1}
  =\inc{t}-\dec{t}
  =\inc{t}-(\decfree{t}+\utt{t}+\ott{t}+\dectails{t}+\deczero{t}).\]
\begin{lemma}\label{lem:first_increase}
  Let $t$ be an arbitrary turn. Let $e_M$ be the \mak-edge of this turn.
  Then,
  \begin{itemize}
  \item[(i)] for every $w\in V$ it holds
    $\inc{t}(w)-\utt{t}(w)\leq(\mu^{p_0\free(t)/\q}-1)\pot_{t-1}(w)$.
  \item[(ii)]$\inc{t}-\utt{t}\leq(\mu^{p_0\free(t)/\q}-1)\pot_{t-1}(e_M)$.
  \end{itemize}
\end{lemma}
\begin{proof}
  $(i)$.
  Let $e_M=\{u,v\}$.
  First note that if $\utt{t}(w)\neq\decheads{t}(w)$,
  it follows that $\utt{t}(w)=\inc{t}(w)$, so there is nothing more to show.
  Otherwise, the term $\inc{t}(w)-\utt{t}(w)$
  describes the change of the potential of $w$
  from the beginning of the turn $t$ to the end of part $1$ of \bre's
  moves in the same turn,
  where we ignore the changes caused by tails of closing edges.
  For $w\notin\{u,v\}$ this is $0$ and we are done.
  So let $w\in\{u,v\}$ and let
  $\first{\deg}_{M,t}(w),\first{\deg}_{B,t}(w),\first{\baldeg}_{t}(w)$ and $\first{\defi}_{t}(w),\first{\pot}_{t}(w)$
  be the \mak-degree, \bre-degree, balanced degree, deficit and potential
  of $w$ after part $1$ of \bre's moves
  (i.e.\ after all closing edges have been claimed).
  To compute the change of the potential of $w$,
  we start by computing the change of its deficit.
  We have
  \begin{align*}
    \first{\defi}_t(w)-\defi_{t-1}(w)
    &=\first{\baldeg}_t(w)-\first{\deg}_{B,t}(w)
      -\baldeg_{t-1}(w)+\deg_{B,t-1}(w)\\
    &=(\first{\baldeg}_t(w)-\baldeg_{t-1}(w))
      -(\first{\deg}_{B,t-1}(w)-\deg_{B,t}(w)).
  \end{align*}
  The first term describes the change of $\baldeg(w)$.
  Since \bre-edges do not influence this value,
  this change is caused solely by $e_M$.
  Due to Lemma~\ref{lem:deficit_single}, this is at most
  $p_0(b-\deg_{M,t-1}(w))$.
  The second term simply describes the number of closing edges
  claimed incident to $w$. Due to Observation~\ref{obs:isolation},
  $t$ is no isolation turn, so
  in case of $w=u$, this is $\deg_{M,t-1}(v)$
  and in case of $w=v$ this is $\deg_{M,t-1}(u)$.
  Together with~\eqref{eq:freeedges} and Remark~\ref{rem:p0} this gives
  \begin{equation}\label{equ}
    \first{\defi}_t(u)-\defi_{t-1}(u)=p_0(\q-\deg_{M,t-1}(u))-\deg_{M,t-1}(v)
    \leq p_0\free(t)
  \end{equation}
  and
  \begin{equation}\label{eqv}
    \first{\defi}_t(v)-\defi_{t-1}(v)=p_0(\q-\deg_{M,t-1}(v))-\deg_{M,t-1}(u)
    \leq p_0\free(t).
  \end{equation}
  This implies
  \begin{align*}
    \inc{t}(w)-\utt{t}(w)
    &=\first{\pot}_{t}(w)-\pot_{t-1}(w)
      =\mu^{\first{\defi}_t(w)/\q}-\pot_{t-1}(w)\\
    &=(\mu^{(\first{\defi}_t(w)-\defi_{t-1}(w))/\q}-1)\pot_{t-1}(w)\\
    &\overset{\mathclap{\eqref{equ},\eqref{eqv}}}{\leq} 
    \hspace{4mm}(\mu^{p_0\free(t)/\q}-1)\pot_{t-1}(w).   
  \end{align*}
  $(ii)$.
  Note that $\inc{t}=\inc{t}(u)+\inc{t}(v)$
  and $\utt{t}=\utt{t}(u)+\utt{t}(v)$,
  so we have
  \begin{align*}
    \inc{t}-\utt{t}
    &=\inc{t}(u)+\inc{t}(v)-(\utt{t}(u)+\utt{t}(v))\\
    &=\inc{t}(u)-\utt{t}(u)+\inc{t}(v)-\utt{t}(v)\\
    &\overset{(i)}{\leq} (\mu^{p_0\free(t)/\q}-1)\pot_{t-1}(u)+(\mu^{p_0\free(t)/\q}-1)\pot_{t-1}(v)\\
    &=(\mu^{p_0\free(t)/\q}-1)\pot_{t-1}(e_M).
  \end{align*}

\end{proof}

\subsection{Critical turns}\label{sec:crit_turns}
Since $\mu\overset{n\rightarrow\infty}\longrightarrow 1$,
with Remark~\ref{rem:p0} and
$n$ big enough we get $\mu p_0<1$.
Fix $\eta\in(0,1-\mu p_0)$ and define the following parts of potential change.
\begin{definition}\label{def:crit}
  For every turn $t$ let
  \[\critdiff{t}:=\inc{t}-\utt{t}-(1-\eta)\decfree{t}\]
  and
  \[\restdiff{t}:=\ott{t}+\dectails{t}+\eta\decfree{t}+\deczero{t}.\]
  We call $t$ \emph{critical}, if $\critdiff{t}> 0$
  and \emph{non-critical} otherwise.
\end{definition}
Note that $\Pot_t-\Pot_{t-1}=\critdiff{t}-\restdiff{t}$. Since $\restdiff{t}\geq 0$,
every turn $t$ with $\Pot_t>\Pot_{t-1}$ is critical.

\begin{lemma}\label{lem:tech1}
  For
  all $x\in\R$ with $x\geq 1$ it holds
  $x(1-\mu^{-1/\q})\geq 1-\mu^{-x/\q}$.
\end{lemma}
\begin{proof}
  We define $g(x):=x(1-\mu^{-1/\q})$
  and $h(x):=1-\mu^{-x/\q}$,
  so we have to show $g(x)\geq h(x)$ for all $x\geq 1$.
  First note that $g(1)=h(1)$,
  so it suffices to show that $g'(x)\geq h'(x)$
  for all $x\geq 1$.
  We have $g'(x)=1-\mu^{-1/\q}$
  and $h'(x)=\mu^{-x/\q}\frac{\ln(\mu)}{\q}$.
  Because for all $x>0$ we have $h''(x)=-\mu^{-x/\q}\left(\frac{\ln(\mu)}{\q}\right)^2<0=g''(x)$,
  it suffices to show that $g'(1)\geq h'(1)$.
  To see this, we use the fact that $e^y-1\geq y$ for all $y\geq 0$,
  so especially
  $\mu^{1/\q}-1\geq\frac{\ln(\mu)}{\q}$.
  If we multiply both sides with $\mu^{-1/\q}$, we get
  \[1-\mu^{-1/\q}\geq \mu^{-1/\q}\frac{\ln(\mu)}{\q}.\]
  Because the left hand side is $g'(1)$ and the right hand side is $h'(1)$,
  we are done.
\end{proof}
The following lemma provides an important characterization of critical turns
by an upper bound for the potential of all edges still unclaimed after the turn.
\begin{lemma}\label{lem:critical_turns}
  Let $t$ be a critical turn with $\free(t)\geq 2$ and
  let $e_M$ be the edge chosen by \mak\ in this turn.
  For every edge $e$ that is still unclaimed after $t$
  it holds
  \[\pot_t(e)< \frac{\mu p_0}{(1-\eta)}\pot_{t-1}(e_M).\]
\end{lemma}
\begin{proof}
  Let $e_M=\{u,v\}$.
  By Lemma~\ref{lem:first_increase}~(ii) we have
  \begin{align*}
    \inc{t}-\utt{t}
    &\leq (\mu^{p_0\free(t)/\q}-1)\pot_{t-1}(e_M)\\
    &=\mu^{p_0\free(t)/\q}(1-\mu^{-p_0\free(t)/\q})\pot_{t-1}(e_M)\\
    &\leq\mu(1-\mu^{-p_0\free(t)/\q})\pot_{t-1}(e_M)
  \end{align*}
  We apply Lemma~\ref{lem:tech1} with $x:=p_0\free(t)$
  (note that due to Remark~\ref{rem:p0} we have 
  $x>\tfrac{8}{3\beta}\free(t)>\tfrac{8}{12}\free(t)\geq\tfrac{16}{12}>1$)
  and get 
  \begin{equation*}
   \inc{t}-\utt{t}\leq \mu p_0\free(t)(1-\mu^{-1/\q})\pot_{t-1}(e_M).
  \end{equation*}
  Because $t$ is a critical turn, we get
  \begin{align*}
    0<\diff{t}
    &=\inc{t}-\utt{t}-(1-\eta)\decfree{t}\\
    &\leq\mu p_0\free(t)(1-\mu^{-1/\q})\pot_{t-1}(e_M)-(1-\eta)\decfree{t},
  \end{align*}
  implying
  \begin{equation}\label{eq:critturn}
    (1-\eta)\decfree{t}< \mu p_0\free(t)(1-\mu^{-1/\q})\pot_{t-1}(e_M).
  \end{equation}
  Now let $e$ be an edge that after turn~$t$ still is unclaimed.
  Then every free edge claimed
  by \bre\ in turn~$t$ has at least a potential of $\pot_{t}(e)$
  because \bre\ iteratively chooses the edge with maximum potential and
  every edge claimed by \bre\ only decreases potential.
  Due to Lemma~\ref{lem:pot_single}~(ii)
  every free edge causes a total potential decrease of at least
  $\pot_t(e)(1-\mu^{-1/\q})$
  and hence we get
  \[\decfree{t}\geq \free(t)\pot_{t}(e)(1-\mu^{-1/\q}).\]
  Together with~\eqref{eq:critturn} this implies
  $\pot_t(e)< \frac{\mu p_0}{(1-\eta)}\pot_{t-1}(e_M)$.
\end{proof}
\subsection{Increase of total potential}\label{sec:increase}
With our strategy we cannot guarantee that $\Pot_t\leq\Pot_{t-1}$ for all turns $t$.
But we will show that each turn $t_0$ at which the potential exceeds $n$
is followed closely by a turn at which the total potential
is at most as big as it was before $t_0$.
So in the long run we obtain a decrease of the total potential,
which will ensure Breaker's win.
\par
Fix constant parameters $\epsa\in(0,1)$, and $\epsb>0$ with
\begin{equation}\label{eq:choiceeps}
  \frac{1-\eta}{(1+\epsb)\mu p_0}>1.
\end{equation}
Recall that this possible, because $\eta<1-\mu p_0$ by the choice of $\eta$.
Define 
\[\numcrit:=\left\lceil\frac{1-\log(1-\epsa)}{\log(1-\eta)-\log(1+\epsb)-\log(\mu p_0)}\right\rceil\]
and note that $\numcrit>0$ due to~\eqref{eq:choiceeps}.
Although $\numcrit$ depends on $n$, it is bounded by constants
because $1<\mu<2$ for $n$ sufficiently big.
Let $t_0$ be a turn with $\Pot_{t_0}>n,\Pot_{t_0-1}\leq n$
and $\Pot_t<2n$ for all $t<t_0$.
Then, $t_0$ is a critical turn
and due to Observation~\ref{obs:freeedges} it holds $\free(t_0)\geq 2$.
Let $e_0=\{u,v\}$ be the edge claimed by \mak\ in this turn
and w.l.o.g.\ let $\pot_{t_0-1}(u)\geq\pot_{t_0-1}(v)$.
We consider three points of time:
\begin{itemize}
\item Let $t_1$ be the first turn after $t_0-1$
  with $\pot_{t_1}(u)\leq(1-\epsa)\pot_{t_0-1}(u)$.
\item Let $t_2$ be the first turn after $t_0$
  with $\pot_{t_2}(w)\geq (1+\epsb)\pot_{s}(w)$
  for some $w\in V$ and some turn $s$ with $t_0\leq s<t_2$.
\item Let $t_3$ be the $\numcrit$-th critical turn after $t_0-1$.
\end{itemize}
If the game ends before the turn $t_i$ is reached,
let $t_i:=\infty$. We set $\tmin:=\min(t_1,t_2,t_3)$
(note that $\tmin=\infty$ is possible)
and aim to prove the following theorem
\begin{theorem}\label{thm:no_increase}
  Let $n$ sufficiently big.
  If the game is not ended before turn~$\tmin$,
  then $\Pot_{t^*}\leq\Pot_{t_0-1}$.
\end{theorem}
Since the proof is quite involved, it is split into several parts.
We start with an observation, that between the turns $t_0$ and $t_2$
the total potential will not exceed $2n$.
\begin{observation}\label{obs:auxobs}
 If $n$ is sufficiently large, 
 for every turn $t$ with $t_0\leq t<t_2$ it holds $\Pot_t<2n$.
\end{observation}
\begin{proof}
 Because $t<t_2$, for every $v\in V$ it holds $\pot_t(v)\leq(1+\epsb)\pot_{t_0}(v)$
  by definition of $t_2$.
  This implies
  \[\Pot_t=\sum_{v\in V}\pot_t(v)\leq\sum_{v\in V}(1+\epsb)\pot_{t_0}(v)
    =(1+\epsilon)\Pot_{t_0}.\]
  By Lemma~\ref{lem:first_increase}~(ii) we have
  \begin{align*}
    \Pot_{t_0}
    &=\Pot_{t_0}-\Pot_{t_0-1}+\Pot_{t_0-1}
      \leq \inc{t_0}-\utt{t_0}+\Pot_{t_0-1}\\
    &\leq \mu^{p_0\free(t_0)/\q}\Pot_{t_0-1}
      \leq \mu\Pot_{t_0-1},
  \end{align*}
  so finally,
  \[\Pot_t\leq(1+\epsilon)\Pot_{t_0}\leq\mu(1+\epsb)\Pot_{t_0-1}\leq\mu(1+\epsb)n<\tfrac{3}{2}\mu n.\]
  For sufficiently large $n$ we have $\mu<\tfrac{4}{3}$ and the proof is complete.
\end{proof}

In the following we assume that the game is not ended
before turn $\tmin$ is reached.
In the next lemma we further refine the characterization of critical turns
from Lemma~\ref{lem:critical_turns}.
We only consider turns between $t_0$ and $t_2$
and prove that the number of critical turns in this interval
affects the maximum possible potential of unclaimed edges exponentially.
\begin{lemma}\label{lem:crit_stacking}
  Let $s$ be a turn with $t_0\leq s\leq \tmin$ and $s< t_2$.
  Let $\crit(s)\in[\numcrit]$ be the number of critical turns
  between $t_0$ and $s$ (including $t_0$ and $s$).
  Then, for every edge $e$ unclaimed after turn $s$ it holds
  \[\pot_s(e)<\left(\frac{(1+\epsb)\mu p_0}{(1-\eta)}\right)^{\crit(s)}2\pot_{t_0-1}(u).\]
\end{lemma}
\begin{proof}
  Via induction over $\crit(s)$.
  \par
  Let $\crit(s)=1$. Recall that $e_0=\{u,v\}$ is the edge claimed by \mak\ in turn $t_0$
  and that $\pot_{t_0-1}(u)\geq\pot_{t_0-1}(v)$.
  Let $e=\{x,y\}$ be an edge unclaimed after turn $s$.
  Because $s<t_2$,
  we know that
  \[\pot_s(e)=\pot_s(x)+\pot_s(y)\leq(1+\epsb)\pot_{t_0}(x)+(1+\epsb)\pot_{t_0}(y)=(1+\epsb)\pot_{t_0}(e)\]
  and because $\free(t_0)\geq 2$, by Lemma~\ref{lem:critical_turns}
  \[(1+\epsb)\pot_{t_0}(e)
    <(1+\epsb)\frac{\mu p_0}{(1-\eta)}\pot_{t_0-1}(e_0)
    \leq\frac{(1+\epsb)\mu p_0}{(1-\eta)}2\pot_{t_0-1}(u).\]
  Now let the claim be true for all $s'$ with
  $\crit(s')=i,i\in[\numcrit-1]$.
  Let $s$ be a turn with $\crit(s)=i+1$.
  Let $s'$ be the last critical turn before $s$
  (if $s$ is critical, let $s'=s$).
  Then $\crit(s'-1)=i$.
  Let $e_M$ be the edge claimed by \mak\ in turn $s'$.
  We get
  \begin{align*}
    \pot_s(e)
    &\leq (1+\epsb)\pot_{s'}(e)\tag{$t<t_2$}\\
    &\leq(1+\epsb)\frac{\mu p_0}{(1-\eta)}\pot_{s'-1}(e_M)
      \tag{Lemma~\ref{lem:critical_turns}}\\
    &\leq(1+\epsb)\frac{\mu p_0}{(1-\eta)}\left(\frac{(1+\epsb)\mu p_0}{1-\eta}\right)^i2\pot_{t_0-1}(u)\tag{IH}\\
    &=\left(\frac{(1+\epsb)\mu p_0}{1-\eta}\right)^{i+1}2\pot_{t_0-1}(u).
  \end{align*}
  Note that for the above application of Lemma~\ref{lem:critical_turns}, 
  we need to ensure that $\free(s')\geq 2$.
  Due to Observation~\ref{obs:freeedges}, it suffices to show that
  $\Pot_t<2n$ for all $t<s'$.
  By choice of $t_0$, we already know that $\Pot_t<2n$ for all $t<t_0$
  and because $s'\leq s<t_2$, for all $t_0\leq t<s'$
  we can apply Observation~\ref{obs:auxobs} and get 
  $\Pot_t<2n$.
\end{proof}

\begin{lemma}\label{lem:crit_increase}
  For every $\epsc>0$, if $n$ is sufficiently big, we have
  \[\sum_{\substack{t_0\leq s\leq \tmin\\s \textnormal{ critical}}}\inc{s}
    \leq 2\numcrit(\mu-1)\pot_{t_0-1}(u)<\epsc\pot_{t_0-1}(u).\]
\end{lemma}
\begin{proof}
  Let $\epsc>0$.
  First note that due to Lemma~\ref{lem:pot_single}~(i)
  \begin{equation}
    \inc{t_0}\leq (\mu-1)\pot_{t_0-1}(e_0)\leq 2(\mu-1)\pot_{t_0-1}(u).
  \end{equation}
  
  Now let $s$ be a critical turn with $t_0<s\leq \tmin$.
  Let $e_M$ be the edge claimed by \mak\ in this turn. 
  We get
  \begin{align*}
    \inc{s}
    &\leq(\mu-1)\pot_{s-1}(e_M)\tag{Lemma~\ref{lem:pot_single}~(i)}\\
    &<(\mu-1)\left(\frac{(1+\epsb)\mu p_0}{(1-\eta)}\right)^{\crit(s-1)}2\pot_{t_0-1}(u)\tag{Lemma~\ref{lem:crit_stacking}}\\
    &\overset{\eqref{eq:choiceeps}}{\leq} (\mu-1)2\pot_{t_0-1}(u).
  \end{align*}
  So for every critical turn $s$ with $t_0\leq s\leq \tmin$ we have
  \begin{equation}
    \inc{s}\leq 2(\mu-1)\pot_{t_0-1}(u).
  \end{equation}
  Because $t\leq t_3$, there are at most $\numcrit$
  critical turns between $t_0$ and $\tmin$,
  so finally we get
  \[\sum_{\substack{t_0\leq s\leq \tmin\\s \textnormal{ critical}}}\inc{s}
    \leq\sum_{\substack{t_0\leq s\leq \tmin\\s \textnormal{ critical}}}2(\mu-1)\pot_{t_0-1}(u)
    \leq 2\numcrit (\mu-1)\pot_{t_0-1}(u).\]
  Recall that
  $(\mu-1)=6\ln(n)\beta/\epsd \q=6\ln(n)\sqrt{\beta}/\epsd \sqrt{n}
  \overset{n\rightarrow\infty}{\longrightarrow} 0$,
  whereas $\numcrit$ is bounded by a constant. 
  So for $n$ sufficiently big, the whole term is smaller than
  $\epsc\pot_{t_0-1}(u)$.
\end{proof}
By definition, $t^*$ always has one of the three values $t_1,t_2,t_3$.
In the following three lemmas we consider all possible cases.
These lemmas combined directly imply Theorem~\ref{thm:no_increase}.
We always assume $n$ to be sufficiently big if needed.
\begin{lemma}\label{lem:t1}
  If $t_1\leq t_2$ and $t_1\leq t_3$, then $\Pot_{t_0-1}\geq \Pot_\tmin$.
\end{lemma}
\begin{proof}
  Let $\epsc\in(0,\eta\epsa)$.
  By assumption $\tmin=\min(t_1,t_2,t_3)=t_1$ and hence, 
  by definition of $t_1$ we have $\pot_{\tmin}(u)\leq(1-\epsa)\pot_{t_0-1}(u)$.
  Let $\bigrest:=\sum\limits_{t_0\leq s\leq \tmin}\diffrest{s}$. Then,
  \begin{align*}
    \Pot_{\tmin}-\Pot_{t_0-1}
    &=\sum_{\mathclap {t_0\leq s\leq t^*}}\Pot_{s}-\Pot_{s-1}
    =\sum_{\mathclap{t_0\leq s\leq t^*}}\diff{s}-\diffrest{s}\\
    &=\sum_{\mathclap{\substack{t_0\leq s\leq \tmin\\s \textnormal{ critical}}}}\diff{s}
    +\underbrace{\sum_{\substack{t_0\leq s\leq \tmin\\s \textnormal{ non-critical}}}\hspace{-0.5cm}\diff{s}}_{\leq 0}-\bigrest\\
    &\leq \sum_{\mathclap{\substack{t_0\leq s\leq \tmin\\s \textnormal{ critical}}}}\diff{s}-\bigrest\\
    &\leq \epsc\pot_{t_0-1}(u)-\bigrest,\tag{Lemma~\ref{lem:crit_increase}}
  \end{align*}
  hence it suffices to show that $\bigrest\geq \epsc\pot_{t_0-1}(u)$.
  We have
  \begin{align*}
    &\epsc\pot_{t_0-1}(u)\\
    &\leq\eta\epsa\pot_{t_0-1}(u)\\
    &\leq\eta(\pot_{t_0-1}(u)-\pot_\tmin(u))\tag{$\tmin=t_1$}\\
    &=\eta\left(\sum_{t_0\leq s\leq \tmin}\dec{s}(u)-\inc{s}(u)\right)\\
    &=\eta\left(\sum_{t_0\leq s\leq \tmin}\utt{s}(u)+\ott{s}(u)
      +\dectails{s}(u)+\decfree{s}(u)+\deczero{s}(u)-\inc{s}(u)\right)\\
    &\leq\eta\left(\sum_{t_0\leq s\leq \tmin}\ott{s}(u)
      +\dectails{s}(u)+\decfree{s}(u)+\deczero{s}(u)\right)
      \tag{$\utt{s} (u)\leq\inc{s} (u)$}\\
    &\leq\sum_{t_0\leq s\leq \tmin}\ott{s}(u)+\dectails{s}(u)+\eta\decfree{s}(u)+\deczero{s}(u)\\
    &\leq\sum_{t_0\leq s\leq \tmin}\restdiff{s}
      =\bigrest.
  \end{align*}
\end{proof}

\begin{lemma}\label{lem:t2}
  If $t_2< t_1$ and $t_2\leq t_3$, then $\Pot_{t_0-1}\geq \Pot_\tmin$.
\end{lemma}
\begin{proof}
  Let $\epsc>0$ with $\epsc\leq\eta(1-\epsa)(1-(1+\epsb)^{-1/p_0})$.
  We have $\tmin=t_2$, so there exists a turn $s_0$ with $t_0\leq s_0< \tmin$
  and a vertex $w\in V$,
  such that $\pot_t(w)\geq(1+\epsb)\pot_{s_0}(w)$.
  Because $\tmin<t_1$, the potential of $u$ was not set to $0$
  and as in the proof of Lemma~\ref{lem:t1}
  it suffices to show that $\bigrest\geq\epsc\pot_{t_0-1}(u)$.
  We start by showing that
  for all turns $t$ with $s_0\leq t \leq \tmin$ it holds
  \begin{equation}\label{prodabsch}
    \pot_{t}(w)\leq\pot_{s_0}(w)\prod_{s_0<s\leq t}\mu^{p_0\free(s)/\q}.
  \end{equation}
  We prove~\eqref{prodabsch} via induction over $t$.
  For $t=s_0$ the claim obviously holds.
  Now let $t>s_0$. Then, $t-1\geq s_0$
  and by Lemma~\ref{lem:first_increase}~(i) we have
  \begin{align*}
    \pot_{t}(w)-\pot_{t-1}(w)
      \leq\inc{t}(w)-\utt{t}(w)
      \leq\pot_{t-1}(w)(\mu^{p_0\free(t)/\q}-1),
  \end{align*}
  so
  \[\pot_{t}(w)\leq\pot_{t-1}(w)\mu^{p_0\free(t)/\q}.\]
  By applying the induction hypothesis we finish the proof of~\eqref{prodabsch}:
  \[\pot_{t}(w)\leq\left(\pot_{s_0}(w)\prod_{s_0<s\leq t-1}\mu^{p_0\free(s)/\q}
    \right)\mu^{p_0\free(t)/\q}
    =\pot_{s_0}(w)\prod_{s_0<s\leq t}\mu^{p_0\free(s)/\q}.\]
  Using~\eqref{prodabsch}, we get
  \[(1+\epsb)\pot_{s_0}(w)\leq\pot_\tmin(w)
      \leq\pot_{s_0}(w)\prod_{s_0<s\leq \tmin}\mu^{p_0\free(s)/\q},\]
      so
      \[
    (1+\epsb)\leq\prod_{s_0<s\leq \tmin}\mu^{p_0\free(s)/\q}
                     =\mu^{\left(p_0\sum_{s_0<s\leq \tmin}\free(s)\right)/\q}\]
                     which,taking the logarithm gives
    \[\sum_{s_0<s\leq \tmin}\free(s)
                     \geq\frac{\q\ln(1+\epsb)}{p_0\ln(\mu)}=:x,\]
  so at least $x$ free edges were claimed by \bre\
  between the turns $s_0$ and $\tmin$.
  Because $\tmin<t_1$, at the whole time from $t_0$ to $\tmin$
  the potential of $u$ is at least $(1-\epsa)\pot_{t_0-1}(u)$.
  Hence, during this time every unclaimed edge incident in $u$
  has a potential of at least $(1-\epsa)\pot_{t_0-1}(u)$,
  so especially every free edge
  claimed by \bre\ has at least this potential
  and, due to Lemma~\ref{lem:pot_single}~(ii),
  causes a decrease of the total potential
  of at least $(1-\epsa)\pot_{t_0-1}(u)(1-\mu^{-\frac{1}{\q}})$.
  Therefore, we get
  \begin{align*}
    \bigrest
    &\geq \eta\sum_{s_0<s\leq \tmin}\decfree{s}\\
    &\geq \eta x (1-\epsa)\pot_{t_0-1}(u)\left(1-\mu^{-\frac{1}{\q}}\right)\\
    &\geq \eta (1-\epsa)\pot_{t_0-1}(u)\left(1-\mu^{-\frac{x}{\q}}\right)
      \tag{Lemma~\ref{lem:tech1}}\\
    &\geq \eta (1-\epsa)\pot_{t_0-1}(u)\left(1-(1+\epsilon)^{-\frac{1}{p_0}}\right)\\
    &\geq \epsc \pot_{t_0-1}(u).
  \end{align*}
\end{proof}

\begin{lemma}\label{lem:t3}
  $t_3\geq\min(t_1,t_2)$.
\end{lemma}
\begin{proof}
  Let us assume that $t_3<\min(t_1,t_2)$.
  Then $\tmin=t_3$, so $\tmin$ is the $\numcrit$-th critical turn after $t_0-1$.
  We apply Lemma~\ref{lem:crit_stacking} to $s=\tmin<t_2$
  and obtain that for every unclaimed edge $e$ after turn $\tmin$ it holds
  \[\pot_\tmin(e)
    <\left(\frac{(1+\epsb)\mu p_0}{(1-\eta)}\right)^{\numcrit}2\pot_{t_0-1}(u)
    \leq (1-\epsa)\pot_{t_0-1}(u)\]
  by the choice of $\numcrit$.
  Since $\tmin<t_1$, we have $\pot_t(u)\geq(1-\epsa)\pot_{t_0-1}(u)$,
  so directly after turn $\tmin$, every unclaimed edge incident in $u$
  has a potential of at least $(1-\epsa)\pot_{t_0-1}(u)$.
  Hence, after turn $\tmin$ there exists no unclaimed edge incident in $u$
  and this implies that the potential of $u$ must have been set to $0$
  at some turn $s$ with $t_0\leq s\leq \tmin$.
  But then $t_1\leq s\leq \tmin=t_3$, a contradiction.
\end{proof}
\begin{proof}[Proof of Theorem~\ref{thm:smallpot}]
  Let $s$ be some turn with $\Pot_t<2n$ for all $t<s$.
  We show that this already implies $\Pot_s<2n$.\par
  If $\Pot_s< n$, there is nothing to show, so let $\Pot_s>n$.
  Let $t_0$ be maximal satisfying $t_0\leq s$ and $\Pot_{t_0-1}\leq n$
  ($t_0$ exists due to Lemma~\ref{lem:start_potential}).
  Define $\tmin$ as in Section~\ref{sec:increase}.
  If $s=\tmin$, we can apply Theorem~\ref{thm:no_increase} and get
  $\Pot_s=\Pot_\tmin\leq\Pot_{t_0-1}\leq n$,
  so we may assume $s<\tmin$. But then, $s<t_2$,
  so we can apply Observation~\ref{obs:auxobs} and obtain that $\Pot_s<2n$.
\end{proof}
\section{Open Questions}
We have narrowed the gap for the threshold bias to $[1.414\sqrt{n},1.633\sqrt{n}]$.
Of course, the question about the exact threshold value remains.
At first sight our strategy still has some unused potential
for improvement, since the secondary goal of preventing \mak\
from building a $\q/2$-star is very restricting.
\bre\ could allow \mak\ to build a few bigger stars,
if at the same time he is able to claim all edges connecting these stars.
For $\q\leq \sqrt{8n/3}$ the strategy still could be used
to prevent \mak\ from building an $\alpha\q$-star for some $\alpha>1/2$.
But it certainly needs some additional variations of the strategy
to prevent \mak\ from connecting stars of size at least $\q/2$.
\section{Acknowledgements}
We thank Joel Spencer for turning our attention to positional games and Ma\l gozata Bednarska-Bzd\c{e}ga for introducing us to the triangle game and
for her valuable comments and suggestions. We would also like to thank Jan Schiemann for
careful proof-reading and his useful comments.
\bibliography{literature}

\section{Appendix}
\subsection{List of variables}
\begin{itemize}
\item $n:$ number of nodes in the game graph
\item $\q:$ number of \bre-edges per turn 
\item $\beta:$ defined as $\beta:=\frac{\q^2}{n}$;
  the strategy in this paper works for 
  $\beta>\frac{8}{3}$.
\item $\firsteps:$ a strictly positive constant. 
\item $\deg_{M,t}(v):$ \emph{\mak-degree of} $v$;
  number of \mak-edges incident in $v$ after turn $t$
\item $\deg_{B,t}(v):$ \emph{\bre-degree of} $v$;
  number of \bre-edges incident in $v$ after turn $t$
\item $\epsd:$ a constant with $0<\epsd<1-\frac{8}{3\beta}$;
  chosen in Section~\ref{sec:potfunc}
\item $\bal(v):$ the \emph{balance} of $v$;
  a measure of the ratio of \mak\ and \bre-edges incident in $v$;
  introduced in Definition~\ref{def:pot}  
\item $p_0:$ the balance of a node without incident \mak\ or \bre-edges;
  introduced in Definition~\ref{def:pot} 
\item $\baldeg(v):$ the \emph{balanced \bre-degree} of $v$;
  introduced in Definition~\ref{def:pot2} 
\item $\defi(v):$  the \emph{deficit} of $v$, exponent in the potential function;
  introduced in Definition~\ref{def:pot2} 
\item $\mu:$ base in the potential function;
  introduced in Definition~\ref{def:pot2} 
\item $\pot(v):$ the \emph{potential} of $v$,
  in part~2 of the strategy \bre\ always claims edges
  $\{u,v\}$ maximizing $\pot(u)+\pot(v)$;
  introduced in Definition~\ref{def:pot2} 
\item $\Pot_t:$ the \emph{total potential} of a turn $t$;
  introduced in Definition~\ref{def:pot2}
\item $\free(t):$ number of \emph{free edges} claimed by \bre\ in turn $t$;
  introduced in Section~\ref{sec:strategy}
\item $\inc{t}(v):$ potential increase of $v$ in turn $t$
\item $\dec{t}(v):$ potential decrease of $v$ in turn $t$
\item $\decfree{t}(v):$ potential decrease of $v$ in turn $t$
  caused by free edges 
\item $\decheads{t}(v):$ potential decrease of $v$ in turn $t$
  caused by closing edges with $v$ as head 
\item $\dectails{t}(v):$ potential decrease of $v$ in turn $t$
  caused by closing edges with $v$ as tail 
\item $\deczero{t}(v):$ potential decrease of $v$ in turn $t$
  caused by claiming the last unclaimed edge of $v$
\item $\utt{t}(v)=\min\{\inc{t}(v),\decheads{t}(v)\}$;
  it holds $\utt{t}(v)+\ott{t}(v)=\decheads{t}(v)$
\item $\ott{v}(v)=\max\{\decheads{t}(v)-\inc{t}(v),0\}$;
  it holds $\utt{t}(v)+\ott{t}(v)=\decheads{t}(v)$
\item $\eta:$ a constant with $0<\eta<1-\mu p_0$;
  introduced in Section~\ref{sec:crit_turns}
\item $\diff{t}:$ main part of the total potential change in turn $t$
  with \mbox{$\diff{t}+\diffrest{t}=\Pot_t-\Pot_{t-1}$};
  introduced in Definition~\ref{def:crit}
\item $\diffrest{t}:$  rest  part of the total potential change in turn $t$,
  with $\diff{t}+\diffrest{t}=\Pot_t-\Pot_{t-1}$;
  introduced in Definition~\ref{def:crit}
\item $\epsa:$ a strictly positive constant;
  introduced in Section~\ref{sec:increase} 
\item $\epsb:$ a strictly positive constant;
  introduced in Section~\ref{sec:increase}
\item $\numcrit:$ a strictly positive value bounded by a constant;
  introduced in Section~\ref{sec:increase}
\item $t_i,i=0,1,2,3:$ certain turns
  considered in Section~\ref{sec:increase}
\item $\tmin=\min(t_1,t_2,t_3)$; introduced in Section~\ref{sec:increase}
\end{itemize}

\subsection{Interpretation of the balance value}
\label{sec:interpretation}
In the following we motivate the definition of the balance value
of a node by giving an `in-game'-example.
Let $v\in V$ with $\deg_M(v)<\frac{\q(1-\epsd)}{2}$
and suppose that \mak\ decides to concentrate on the node $v$,
i.e., from this moment on he will claim all of his edges incident in $v$
as long as there are unclaimed edges incident in $v$.
Moreover suppose that \bre's aim,
besides closing all \mak-paths of length $2$,
is to keep $\deg_M(v)$ below $\frac{\q(1-\epsd)}{2}$.
To achieve this,
he must claim a certain amount of edges incident in $v$ himself.
Denote this amount by $B_v$.
Let $T$ denote the number of turns that \mak\ needs to raise
$\deg_M(v)$ above $\left\lceil\frac{\q(1-\epsd)}{2}\right\rceil$.
Then $B_{\textnormal{total}}:=Tb$ is the number of edges that \bre\
can claim before $\deg_M(v)\geq\left\lceil\frac{\q(1-\epsd)}{2}\right\rceil$.
But there is a certain number $C$ of edges that \bre\ has to claim
at different places, not incident in $v$, to close new \mak-paths.\par
Setting $A:=B_{\textnormal{total}}-C$ as the amount of available \bre-edges,
the term $\frac{B_v}{A}$ represents the fraction of \emph{available} \bre-edges
necessary to prevent \mak\ from building a $\q/2$-star.
We will show that $\bal(v)$ is an approximation of $\frac{B_v}{A}$,
hence it is a measure for the `danger' of $v$:
The smaller $\frac{B_v}{A}$ is, the less attention \bre\ has to spend
to the node $v$. If $\frac{B_v}{A}>1$, this means
that \bre\ cannot achieve his goal of keeping $\deg_M(v)$ below $\q/2$.\par
For $f,g:\N\rightarrow \R$ we write $f\sim g$ if and only if
$\lim\limits_{n\rightarrow\infty}\frac{f(n)}{g(n)}=1$.
We will close this subsection by showing that $\bal(v)\sim\frac{B_v}{A'}$
for some $A'\leq A$.
To prevent \mak\ from building a $\frac{\q(1-\epsd)}{2}$-star at $v$,
at the end of the game \bre\ must possess at least $n-\frac{\q(1-\epsd)}{2}$
edges incident in $v$.
Hence, the number of edges still to claim is
$B_v=n-\frac{\q(1-\epsd)}{2}-\deg_B(v)\sim n-\deg_B(v)$.
Because \mak\ claims one edge per turn and concentrates on $v$,
we get $T=\frac{\q(1-\epsd)}{2}-\deg_M(v)$
and $B_{\textnormal{total}}=\frac{\q^2(1-\epsd)}{2}-\q\deg_M(v)$.
The exact value of $C$ depends on the choices of \mak\ and on how many
closing edges are already owned by \bre.
If we assume that all closing edges are previously unclaimed,
we can upper bound $C$ by 
\begin{align*}C'&:=\sum\limits_{i=\deg_M(v)}^{\lceil \q(1-\epsd)/2\rceil -1}i\\
                &=\frac{(\lceil \q(1-\epsd)/2\rceil-1)\cdot\lceil \q(1-\epsd)/2\rceil}{2}
                  -\frac{(\deg_M(v)+1)\deg_M(v)}{2}\\
                &\sim \frac{\q^2(1-\epsd)^2}{8}-\frac{\deg_M(v)^2}{2}. 
\end{align*}
Finally, for $A':=B_{\textnormal{total}}-C'\leq A$ we get
\begin{align*}\frac{B_v}{A'}
  &=\frac{B_v}{B_{\textnormal{total}}-C'}\\
  &\sim \frac{n-\deg_B(v)}{\frac{\q^2(1-\epsd)}{2}-b\deg_M(v)
    -\left(\frac{\q^2(1-\epsd)^2}{8}-\frac{\deg_M(v)^2}{2}\right)}\\
  &=\frac{8(n-\deg_B(v))}
    {\q^2(1-\epsd)(3+\epsd)-4\deg_M(v)(2\q-\deg_M(v))}
    =\bal(v).
\end{align*}
\end{document}